\documentclass[reqno,12pt,letterpaper]{amsart}
\usepackage{amsmath,amssymb,amsthm,graphicx,mathrsfs,url}
\usepackage{enumitem}

\usepackage{graphicx}
\usepackage{amsmath, amscd}
\usepackage[usenames,dvipsnames]{color}
\usepackage[colorlinks=true, linkcolor=Red, citecolor=Green]{hyperref}

\newtheorem{theo}{Theorem}[section]
\newtheorem{prop}{Proposition}[section]

\newtheorem{corr}{Corollary}[section]
\newtheorem{lemm}[prop]{Lemma}
\theoremstyle{definition}

\numberwithin{equation}{section}

\newcommand{\R}{\mathbb{R}}
\newcommand{\N}{\mathbb{N}}

\newcommand{\C}{\mathbb{C}}

\newcommand{\e}{\mathrm{e}}

\newcommand{\Pbf}{\mathbf{P}}

\let\Re=\Real
\def\Res{{\rm Res}\:}

\DeclareMathOperator{\id}{Id}

\let\Im=\Imag
\def\C{\mathbb {C}}
\def\ep{\epsilon}

\title[Dynamical zeta function]{Dirichlet dynamical zeta function for billiard flow}

\author[V.Petkov]{Vesselin Petkov}
\address{Université de Bordeaux, Institut de Mathématiques de Bordeaux, 351, Cours de la Libération, 33405 Talence, France}
\email{petkov@math.u-bordeaux.fr}

\begin{document}

\maketitle
\begin{abstract} We study the Dirichlet dynamical zeta function $\eta_D(s)$ for billiard flow corresponding to several strictly convex disjoint obstacles. For large $\Re s$ we have $\eta_D(s) =\sum_{n= 1}^{\infty} a_n e^{-\lambda_n s}, \: a_n \in \R$ and $\eta_D$ admits a meromorphic continuation to $\C$. We obtain some conditions of the frequencies $\lambda_n$ and some sums of coefficients $a_n$ which imply that $\eta_D$ cannot be prolonged as entire function.

\end{abstract}
\section{Introduction}
Let $D_1, \dots, D_r \subset \R^d,\: {r \geqslant 3},\: d \geqslant 2,$ be compact strictly convex disjoint obstacles with $C^{\infty}$ smooth boundary and let $D = \bigcup_{j= 1}^r D_j.$  We assume that every $D_j$ has non-empty interior and throughout this paper we suppose the following non-eclipse condition
\begin{equation}\label{eq:1.1}
D_k \cap {\rm convex}\: {\rm hull} \: ( D_i \cup D_j) = \emptyset, 
\end{equation} 
for any $1 \leqslant i, j, k \leqslant r$ such that $i \neq k$ and $j \neq k$.
Under this condition all periodic trajectories for the billiard flow in $\Omega  = \R^d \setminus \mathring{D}$ are ordinary reflecting ones without tangential intersections to the boundary of $D$.  We consider the (non-grazing) billiard flow {$\varphi_t$} (see \cite[Section 2]{Petkov2024} for the definition).  Next the periodic trajectories will be called periodic rays. For any periodic ray $\gamma$, denote by  $\tau(\gamma) > 0$ its period, by $\tau^\sharp(\gamma) > 0$ its primitive period, and by $m(\gamma)$ the number of reflections of $\gamma$ at the obstacles. Denote by  $P_\gamma$ the associated linearized Poincar\'e map (see \cite[Section 2.3]{petkov2017geometry} for the definition).

 Let $\mathcal{P}$ be the set of all oriented periodic rays. The counting function of the lengths of primitive periodic rays $\Pi$ satisfies 
\begin{equation} \label{eq:1.2}
\sharp\{\gamma \in \Pi:\: \tau^\sharp(\gamma) \leqslant x\} \sim \frac{\e^{h x}}{h x} , \quad x \to + \infty,
\end{equation}
for some $h > 0$ (see for instance, \cite[Theorem 6.5] {Parry1990} for weak-mixing suspension symbolic flows). Thus there exists an infinite number of primitive periodic trajectories and for every small $\ep > 0$ we have  the estimate
\begin{equation}\label{eq:1.3}
e^{(h- \ep) x} \leq  \sharp\{\gamma \in \mathcal P:\: \tau(\gamma) \leqslant x\} \leqslant \e^{ (h + \ep) x}, \: x > C_{\ep}.
\end{equation}
Moreover, for some positive constants $C_1, d_1, d_2$ we have (see for instance \cite[Appendix] {petkov1999zeta})
\begin{equation} \label{eq:1.4}
C_1 \e^{d_1 \tau(\gamma)} \leqslant |\det(\mathrm{Id} - P_{\gamma})| \leqslant \e^{d_2 \tau(\gamma)}, \quad \gamma \in \mathcal{P}.
\end{equation}
By using these estimates, we define for $\Re(s) \gg 1$ the Dirichlet dynamical zeta function $\eta_D(s)$ by 
\[
 \eta_\mathrm{D}(s) = \sum_{\gamma \in \mathcal{P}} (-1)^{m(\gamma)} \frac{ \tau^\sharp(\gamma) \e^{-s\tau(\gamma)}}{|\det(\mathrm{Id}-P_\gamma)|^{1/2} },
\]
where the sums run over all oriented periodic rays. 
 This zeta function is important for the analysis of the distribution of the scattering resonances related to the Laplacian in $\R^d \setminus \bar{D}$ with Dirichlet  boundary conditions on $\partial D$ (see \cite[\S1]{chaubet2022} for more details). Denote by $\sigma_a \in \R, \sigma _c \in \R$  the abscissa of absolute convergence and the abscissa of convergence of $\eta_D$, respectively.

It was proved in \cite[Theorem 1 and Theorem 4]{chaubet2022} that $\eta_D$  admits a meromorphic continuation to $\C$ with simple poles and integer residues. 
 On the other hand,  for $d = 2$  \cite{stoyanov2001spectrum} and for $d\geqslant 3$ under some conditions  \cite{stoyanov2012non} Stoyanov proved that there exists $\varepsilon > 0$ such that $\eta_{\mathrm D}(s)$ is analytic for $\Re s \geqslant \sigma_a - \varepsilon$.

    There is a conjecture that $\eta_D$ cannot be prolonged as {\it entire function}.  This conjecture was established for obstacles with real analytic boundary (see \cite[Theorem 3]{chaubet2022}) and for obstacles with sufficiently small diameters \cite{ikawa1990zeta}, \cite{stoyanov2009poles} and $C^{\infty}$ smooth boundary. If $\eta_D(s)$ is not an entire function, then we obtain  two important corollaries:
    
 (i)    $\eta_D$  has infinite number of poles in some strip $\{ z \in \C: \Re z \geq \beta\}$ (see \cite[Section 3]{Petkov2024} for a lower bound of the counting function of poles),
 
 (ii)  The modified Lax-Phillips conjecture (MLPC) for scattering resonances introduced by Ikawa \cite{ikawa1990poles} holds. (MLPC) says that there exists a strip $\{z \in \C: 0 < \Im z \leq \alpha\}$ containing an infinite number of scattering resonances for Dirichlet Laplacian in $\R^d \setminus \bar{D}$ (see \cite[Section 1]{chaubet2022} for definitions and more precise results).
  
  Let $\rho \in C_0^{\infty}(\R; \R_+)$ be an even function with ${\rm supp}\: \rho \subset \left[-1, 1\right]$ such that
$\rho(t) > 1 \quad \text{if} \quad |t| \leqslant 1/2.$ Let $(\ell_j)_{j \in \N}$ and $(m_j)_{j \in \N}$ be sequences of positive numbers such that $\ell_j \geqslant d_0 = \min_{k \neq m} {\rm dist}\: (D_k, D_m)> 0, \: m_j \geqslant \max\{1, \frac{1}{d_0}\}$ and let $\ell_j \to \infty,\: m_j\to \infty$ as $j \to \infty$. Set
$\rho_j(t) = \rho(m_j (t- \ell_j)), \quad t \in \R,$
and introduce the distribution $\mathcal F_\mathrm D(t) \in  {\mathcal S}'(\R^+)$ by
\begin{equation*} 
\mathcal F_\mathrm D(t) = \sum_{\gamma \in \mathcal P} \frac{(-1)^{m(\gamma)} \tau^{\sharp}(\gamma)  \delta(t - \tau(\gamma))}{|\det(I - P_{\gamma})|^{1/2}}.
\end{equation*}
We have the following
\begin{prop} The function $ \eta_\mathrm D(s)$ cannot be prolonged as an entire function of $s$ if and only if  there exists $\alpha_0 > 0$ such that for any $\beta > \alpha_0$ we can find sequences  $(\ell_j), (m_j)$ with $\ell_j \nearrow \infty$ as $j \to \infty$  such that for all $j \geqslant 0$ one has $ e^{\beta \ell_j} \leqslant m_j \leqslant \e^{2 \beta \ell_j}$ and
\begin{equation} \label{eq:1.5}
|\langle \mathcal F_\mathrm D, \rho_j \rangle | \geqslant \e^{- \alpha_0 \ell_j}.
\end{equation}
\end{prop}   
  More precisely, if $\eta_D$ cannot be prolonged as entire function, the existence of  sequences $(\ell_j),\:(m_j)$ with the above properties has been proved by Ikawa \cite[Prop.2.3]{ikawa1990poles}, while in the proof of Theorem 1.1 in \cite{Petkov2024}  it was established that if such sequences exist, the function $\eta_D$ has an infinite number of poles. 
  
        The conditions of Proposition 1.1 are difficult to verify. The purpose of this Note is to find other conditions which imply that $\eta_D$ cannot be prolonged as  entire function. For this purpose we exploit  the local trace  formula (see \cite[Theorem 2.1]{Petkov2024}) and the summability by typical means of Dirichlet series introduced by Hardy and Riesz \cite{hardy1964} (see also \cite[Section 2]{Defant2022}).     
  It is convenient to  write $\eta_D(s)$ as a Dirichlet series
  \begin{equation} \label{eq:1.6}
   \eta_D(s) = \sum_{n=1}^\infty a_n e^{-\lambda_n s},\: \Re s \gg 1,
   \end{equation}
   where the frequencies are arranged as follows
   \begin{equation*}
   0 < \lambda_1 < \lambda_2 <...< \lambda_n <...
   \end{equation*}
   and
    \begin{equation} \label{eq:1.7}
  a_n = \sum_{\gamma \in \mathcal P, \tau(\gamma)  = \lambda_n}\frac{(-1)^{m(\gamma)} \tau^\sharp(\gamma) }{|\det(\mathrm{Id}-P_\gamma)|^{1/2}} .
\end{equation}
 It is well known  that 
 $$\sigma_c \geq \sigma_a - \limsup_{n \to \infty} \frac{\log n}{\lambda_n}= \sigma_a - h$$
  (see for instance, \cite[Theorem 9]{hardy1964}).  Since $\sigma_a > - \infty$,  one deduces $\sigma_c > - \infty.$

  Our main result is the following
      \begin{theo} Suppose $\sigma_c < 0$. Assume that there exist constants $C > 0, \:\delta > h + 1,\: -\gamma < \sigma_c$ and an increasing sequence $m_j \nearrow \infty$ such that
      \begin{equation} \label{eq:1.8}
      \lambda_{m_j} - \lambda_{m_j- 1} \geq Ce^{-\delta \lambda_{m_j}},
      \end{equation}
      \begin{equation}\label{eq:1.9}
      |\sum_{n \geq m_j} a_n| \geq e^{-\gamma \lambda_{m_j}}. 
      \end{equation}
            Then $\eta_D(s)$ cannot be prolonged as entire function.
   \end{theo}   
    The condition $\sigma_c < 0$ is not a restriction since if $\sigma_c \geq 0,$ the Dirichlet series
 \[ \eta_D(s + \sigma_c + 1) = \sum_n (a_n e^{-\lambda_n (\sigma_c + 1)}) e^{-\lambda_n s}= \sum_n b_n e^{-\lambda_n s}\]
is convergent for $\Re s > -1$, hence it has a negative abscissa of convergence $\sigma_b$ and $\eta_D(s + \sigma_c + 1)$ is entire if and only if $\eta_D(s)$ is entire.  Moreover, in the proof of Theorem 1.1 (see section 3), assuming $\eta_D$ entire,  one has the property 
 \[ \forall A < \sigma_c, \: \exists C_A > 0,\: \: |\eta_D(s)| \leq C_A( 1 + |\Im s|), \: \Re s \geq A\]
 which is satisfied also for $\eta_D(s + \sigma_c + 1)$ with another constants $B_A$. Thus we may apply Theorem 1.1 if instead of (\ref{eq:1.9}) one has the estimate
\begin{equation} \label{eq:1.10}
 |\sum_{n \geq m_j} b_n| \geq e^{-\gamma_1 \lambda_{m_j}},\: -\gamma_1 < \sigma_b .
 \end{equation}
 
   The assumptions on $\lambda_n$ and $a_n$ in Theorem 1.1  are satisfied if {\it Bohr condition} (see for instance, \cite[\S3.13]{Defant2022})
  \[ (BC)\:\:\: \:\:\:\exists C_1 > 0,\: \exists \ell > 0, \:\: \forall n > 0, \: \lambda_{n + 1} - \lambda_n \geq C_1e^{-\ell \lambda_n} \]
   holds. Indeed, it is well known that in the case $\sigma_c < 0$, one has the representation
   \[ \sigma_c = \limsup_{n \to \infty}  \frac{ \log |\sum_{n \geq m} a_n|}{\lambda_m} .\]
For small $\ep > 0$ this implies the existence of a sequence $m_j \nearrow \infty$ such that
\[ |\sum_{n \geq m_j} a_n| \geq e^{(\sigma_c - \ep) \lambda_{m_j} }
\]
and we obtain (\ref{eq:1.9}) with $-\gamma = \sigma_c - \ep.$ 

The condition (BC) is very restrictive. The advantage of Theorem 1.1 is that (\ref{eq:1.8}) is always satisfied (see Section 3) for infinite number of frequencies $\lambda_{m_j- 1}, \: \lambda_{m_j}$ and the separation by $e^{-\delta \lambda_j}$ of some frequencies $\lambda_{m_j}$ only on  the left is less restrictive than a separation  of all frequencies on both sides.  Applying Theorem 1.1, we obtain the following 
\begin{corr} Suppose $\sigma_c < 0.$ Then if
\begin{equation} \label{eq:1.11}
 \liminf_{m \to \infty} \frac{ \log |\sum_{n \geq m} a_n|}{\lambda_m} > - \infty,
 \end{equation}
 the function $\eta_D(s)$ cannot be prolonged as entire function.
\end{corr}
In Section 4 for $\delta > h + 2$ we construct intervals $I(\lambda_k, \delta) \subset [b, b+1], \: b \geq b_0$ with {\it clustering frequencies} and we obtain Corollary 4.1. We have infinite number of such intervals. Moreover, under some geometrical assumptions described in \cite[Section 8]{petkov2012}  the number of such intervals is exponentially increasing when $b \to \infty.$ Finally, assuming that the coefficients $a_n$ have a lower bound (\ref{eq:4.4}), we prove Corollary 4.2 and we show that for every interval $I(\lambda_k, \delta)$ we have 4 possibilities concerning the  behaviour of the corresponding sums. For 3 of these 4 possibilities it is possible to find frequencies satisfying (\ref{eq:1.8}), (\ref{eq:1.9}) (see Proposition 4.1).

 The paper is organised as follows. In Section 2 we recall the local trace formula for $\eta_D.$ Assuming $\eta_D$ entire, we deduce the estimates (\ref{eq:2.3}). This makes possible to prove that that the abscissa of $k-$ summability $\sigma_k$ of $\eta_D$ is $-\infty.$ In Section 3 we prove Theorem 1.1. Section 4 is devoted to intervals $I(\lambda_k, \delta)$ with clustering frequencies and the constructions of frequencies satisfying (\ref{eq:1.8}) and (\ref{eq:1.9}).
 
  \section*{Acknowledgements} We would like to thank the referee for his/her valuable remarks and suggestions. 

\section{Summation by typical means of $\eta_D$}

In this section we apply the results of \cite{dyatlov2016pollicott}, {\cite[\S6.1]{jin2023resonances} and \cite{Petkov2024} for vector bundles. For our exposition we need only the local trace formula  containing the poles of the meromorphic continuation of cut off resolvents ${\bf 1}_{\tilde{V}_u} ( -i \Pbf _{k, \ell}- s)^{-1}{\bf 1}_{\tilde{V}_u}$ of some operators 
\[-i \Pbf_{k, \ell,q},\: 0 \leq k \leq d,\: 0 \leq \ell \leq d^2 - d, \: q = 1, 2.\] 
Here $\tilde{V}_u$ us a neighborhood of the trapping set $\tilde{K}_u.$ 
The precise definitions of $\Pbf_{k, \ell, q}, \: \tilde{K}_u$ and the corresponding setting are complicated and they are not necessary for the analysis below and we prefer to refer to \cite[Section 2]{Petkov2024} for the corresponding definitions and details. Denote by  $\Res (- i \Pbf_{k, \ell, q})$ the set of the poles of the meromorphic continuation of the corresponding cut off resolvents.

 For every $ A > 0$ and any $0 < \epsilon \ll 1$ we have the following local trace formula (see \cite[Theorem 2.1]{Petkov2024})

\begin{eqnarray} \label{eq:2.1}
\sum_{k = 0}^d \sum_{\ell = 0}^{d^2 - d} \sum_{\mu \in \Res(- i \Pbf_{k, \ell, 2}), \Im \mu > -A} ( -1)^{k + \ell}e^{- i\mu t}\nonumber\\
-\sum_{k = 0}^d \sum_{\ell = 0}^{d^2 - d} \sum_{\mu \in \Res(- i \Pbf_{k, \ell, 1}), \Im \mu > -A} ( -1)^{k + \ell}e^{- i\mu t} \\
+ F_{A}(t)\nonumber = \mathcal F_D(t),\: t > 0.
\end{eqnarray}
Here $F_A(t) \in {\mathcal S}'(\R)$ is  supported in $[0, \infty)$, the Laplace-Fourier transform $\hat{F}_A(\lambda)$ of  $ F_A(t)$  is holomorphic for $\Im \lambda < A - \epsilon$ and satisfies the estimate
\begin{equation}\label{eq:2.2}
|\hat{F}_A(\lambda) | = {\mathcal O}_{A, \ep} (1 + |\lambda|)^{2d^2 + 2d - 1 + \epsilon},\: \Im \lambda < A - \epsilon.
\end{equation}
 
 Notice that the poles in $\Res(-i \Pbf_{k, \ell, q})$ are simple with positive integer residues \cite[Theorem 1]{chaubet2022}. For the sums with fixed $q$ the cancellations in (\ref{eq:2.1}) could appear only between  the terms with $k + \ell$ odd and $k + \ell$ even.  On the other hand, taking the difference of sums with $q = 2$ and $q = 1$ we obtain more cancelations. 
 
 If the following we assume that $\eta_D$ can be a prolonged as entire function. In particular, 
 \[\eta_D( -i \lambda) = \langle \mathcal F_D , e^{i t \lambda}\rangle = \sum_{\gamma \in \mathcal P} \frac{(-1)^{m(\gamma)}\tau^\sharp(\gamma) e^{i\lambda \tau(\gamma)}} {|\det(\id -  P_{\gamma})|^{1/2}},\: \Im \lambda \geq C \gg 1\]
 admits an analytic continuation for $\Im \lambda < C.$ For fixed $A > 0$ the function $\eta_D( - i \lambda)$ has no poles $ \mu$ with $\Im \mu > - A$ and in (\ref{eq:2.1}) all terms involving poles will be canceled. Consequently, from (\ref{eq:2.1}) we obtain 
 \[\eta_D(- i \lambda) = \hat{F}_A(-\lambda), \: \Im \lambda > - A + \epsilon.\]
  Setting $- i\lambda = s = \sigma + i t,\: \sigma \in \R, \:t \in \R$, this implies
 \begin{equation} \label{eq:2.3}
 |\eta_D(s)| \leq C_A(1 + |s|)^{2d^2 + 2d - 1} \leq B_A(1+ |t|)^{2 d^2 + 2d - 1},\: \sigma \geq -A+\epsilon.
\end{equation}
Here we used the fact that $|\eta_D(s)|$ is bounded for $\sigma \geq C_0 > 0$ with sufficiently large $C_0 > 0$ and
$|s| \leq \max\{A, C_0\} + |t|$ for $- A \leq \sigma \leq C_0.$ We may apply the above argument for every $A > 0,$ so the bound (\ref{eq:2.3}) holds for every $A > 0$ with constants $B_A$ depending of $A$. The crucial point is that the power $2 d^2 + 2d - 1$ is {\it independent} of $A$. 

Applying the Phragmént- Lindel\"{o}f principle for entire function $\eta_D(s)$ in the strip
\[ \{z \in \C: -A + \epsilon \leq \Re z \leq C_0\},\]
one deduces
\[|\eta_D(\sigma + it) | \leq D_{\sigma, A} ( 1 + |t|)^{\kappa(\sigma)},\: -A + \epsilon \leq \sigma \leq C_0\]
with 
\[ \kappa(\sigma) = \frac{C_0 - \sigma}{C_0 + A - \epsilon} (2d^2 + 2d - 1), \: -A  +\epsilon \leq \sigma  \leq C_0.\]
For fixed $\sigma$, taking $A$ sufficiently large we obtain for every small $0 < \nu \ll 1$ the estimate
\begin{equation} \label{eq:2.4}
|\eta_D(\sigma + it) | \leq B_{\sigma, \nu} ( 1 + |t|)^{\nu},\: \sigma \leq C_0.
\end{equation}

Next we recall the {\it summation by typical means} of Dirichlet series (see \cite[Section IV, \S2]{hardy1964}, \cite[Section 2]{Defant2022} for more details). For $k > 0$ consider 
\[ C_{\lambda}^k(u) = \sum_{\lambda_n < u} (u - \lambda_n)^k a_n e^{-\lambda_n s} .\]
We say that the series $\sum_{n= 1}^{\infty} a_n e^{-\lambda_n s}$ is $(\lambda, k)$ summable if
\[ \lim_{u \to \infty} \frac{C_{\lambda}^k(u)}{u^k} = C.\]
There exists a number $\sigma_k$ such that the series is $(\lambda, k)$ summable for $\sigma > \sigma_k$ and not $(\lambda, k)$ summable for $\sigma < \sigma_k$ (see \cite[Theorem 26]{hardy1964}).The number $\sigma_k$ is called abscissa of $k$- summability of the series.
We will apply the following 
\begin{theo}[Theorem 41, \cite{hardy1964}] Suppose that the series $f(s) = \sum_{n = 1}^{\infty} a_n e^{-\lambda_n s}$ admits an analytic continuation for $\sigma > \eta.$ Suppose further that $k$ and $k'$ are positive numbers such that $k' < k$ and for all small $\delta$ we have
\[ |f(s)| \leq \mathcal C_{\delta}( 1 + |s|)^{k'}\]
uniformly for $ \sigma \geq \eta + \delta > \eta.$ Then $f(s)$ is $(\lambda, k)$ summable for $\sigma > \eta.$
\end{theo}
 
In fact, the above theorem in \cite{hardy1964} is given without proof. The reader may consult Corollary 3.8 and Corollary 3.9 in \cite{Defant2022} for a proof and other results related to Theorem 2.1 and  $(\lambda, k)$ summability.
 The estimates (\ref{eq:2.4}) combined with Theorem 2.1 imply the following
  \begin{prop}
  If $\eta_D(s)$ can be prolonged as entire function, for every $k > 0$ the  Dirichlet series $(\ref{eq:1.6})$ has abscissa of $k-$summability $\sigma_k = - \infty$.  
  \end{prop}
  
   \section{Proof of Theorem $1.1$}
\def\J{J_{\delta}}
\def\G{G_{\delta}}
Throughout this section we assume that $\eta_D(s)$ can be prolonged as entire function.
  Choose $\delta >  h + 2$. 
 First, it is easy to see that in every interval $[b, b+ 1], b \geq b_0 \gg 1$ we have subintervals $[\alpha, \beta] \subset [b + b+ 1]$ of length greater than $e^{-\delta b}$ which does not contain frequencies. It is sufficient to write $[b, b+ 1]$ as an  union of $e^{\delta b}$ intervals of length $e^{-\delta b}$ and to use the bounds (\ref{eq:1.3}). 
  
    We have the following simple
   \begin{lemm} Fix $0 < \ep < 1/2$ and $0 < \eta < \frac{\ep}{12(1+\ep)}.$ There exists $b_0 \geq \max\{ 3/h, 1\}$ depending of $\eta$ so that for $\alpha \geq b_0$ we have
  \begin{equation} \label{eq:3.1}
    \sharp\{ \gamma \in \Pi: \alpha \leq \tau^{\sharp}(\gamma) \leq \alpha + \ep\}   > \frac{ \ep( 1- \eta)  e^{\alpha h}}{3(\alpha + \ep)}.   
\end{equation}   
  \end{lemm} 
\begin{proof}  For $ x \geq b_0(\eta) \gg 1$ the asymptotics (\ref{eq:1.2}), imply the estimates
   \[\frac{e^{hx}}{hx} ( 1- \eta) \leq \sharp\{ \gamma \in \Pi: \tau^{\sharp}(\gamma) \leq x\} \leq \frac{e^{hx}}{hx} (1 + \eta).\]
   Therefore for $\alpha \geq b_0(\eta)$ we obtain
   \[\sharp\{ \gamma \in \Pi: \alpha \leq \tau^{\sharp}(\gamma) \leq \alpha + \ep\}  \geq \frac{e^{h(\alpha + \ep)}}{h(\alpha + \ep)}( 1 - \eta)  - \frac{e^{h \alpha}}{h\alpha }( 1 + \eta)\]
  \[> \frac{( 1- \eta) e^{\alpha h}}{h(\alpha + \ep) }\Bigl[ 1 + \ep h - \frac{(\alpha + \ep)(1+\eta)}{\alpha( 1 - \eta)}\Bigr].\]
  On the other hand, we have $\frac{1}{\alpha} \leq \frac{h}{3}$ and
  \[4 \eta \leq \frac{\ep h}{3(1 + \ep)} \leq  \frac{\ep h}{3( 1 + \frac{\ep}{\alpha})}.\]
  Then
  \[\frac{(\alpha + \ep)(1+\eta)}{\alpha( 1 - \eta)}  =\Bigl(1 + \frac{\ep}{\alpha}\Bigr) \Bigl( 1 + \frac{2\eta}{1 - \eta}\Bigr) \leq (1+ \frac{\ep}{\alpha})( 1 + 4\eta) \leq 1 + \frac{2\ep h}{3}\]
  and one deduces (\ref{eq:3.1}). \end{proof}

 {\it Proof of Theorem 1.1.}  We start with the formula for the abscissa of $k-$summability $\sigma_k < 0$ in the case when $k \in \N$ is an integer established by Kuniyeda \cite[Theorem E]{Kuniyeda2016}. More precisely,  we have
\begin{equation}\label{eq:3.2}
\sigma_k = \limsup_{u \to \infty} \frac{\log |R^k(u)|}{u^k},
 \end{equation}
 where $R^k(u) = \sum_{\lambda_n> u} a_n(\lambda_n- u)^k.$ We are interesting in the case $k = 1.$ Let $\delta, -\gamma$ be the constants given in (\ref{eq:1.8}), (\ref{eq:1.9}), respectively.
By Proposition 2.1, for $\eta_D(s)$ we have $\sigma_1 = - \infty$. We fix $\gamma_1 > 0$ so that  $- \gamma_1 < - \delta- \gamma - 1.$
Then (\ref{eq:3.2}) implies that
there exists $M = M(\gamma_1) > 1$ such that 
\[|R(u)| = |\sum_{\lambda_n> u} a_n(\lambda_n- u)| \leq e^{-\gamma_1 u}, \: \forall u \geq M.\]
Let 
\[ \lambda_{m_j} - \lambda_{m_j- 1} \geq Ce^{-\delta \lambda_{m_j}},\: \lambda_{m_j- 2} \geq M,\: |\sum_{n \geq m_j} a_n| \geq e^{-\gamma \lambda_{m_j}}.\]
Obviously, for $M$ large by using  (\ref{eq:3.1}), we get $\lambda_{m_j} - \lambda_{m_j - 1} < 1.$

Choose $u_{m_j - 1}, u_{m_j}$ so that $ \lambda_{m_j - 2} < u_{m_j- 1} < \lambda_{m_j- 1} < u_{m_j} < \lambda_{m_j}$ and write
\[ R(u_{m_j- 1}) - R(u_{m_j}) = a_{m_j- 1} (\lambda_{m_j- 1} - u_{m_j- 1})+ (u_{m_j} - u_{m_j- 1}) \sum_{\lambda_n > u_{m_j}} a_n.\]
We choose $\lambda_{m_j - 1} - u_{m_j-1}= \epsilon_j \ll 1$ sufficiently small  to arrange
\[|a_{m_j- 1}| (\lambda_{m_j- 1} - u_{m_j- 1}) \leq e^{- \gamma_1 \lambda_{m_j}}.\]
(Exploiting (\ref{eq:1.4}), we obtain an upper bound $|a_n| \leq e^{c \lambda_n}, \: \forall n$ with $c > 0$ independent of $\lambda_n$,  but this is not necessary for the estimation above.)
Next
\[u_{m_j} - u_{m_j - 1} = (u_{m_j} - \lambda_{m_j} )+ (\lambda_{m_j} - \lambda_{m_j - 1}) + (\lambda_{m_j - 1} - u_{m_j - 1}).\]
Taking $\epsilon_j$ very close to 0, if it is necessary, we choose  $u_{m_j}=   \lambda_{m_j}  - \epsilon_j,$ and deduce
\[ u_{m_j} - u_{m_j- 1} = \lambda_{m_j} - \lambda_{m_j - 1} \geq Ce^{-\delta \lambda_{m_j}}.\]
Then
\[Ce^{(-\delta - \gamma) \lambda_{m_j}} \leq (u_{m_j} - u_{m_j- 1}) |\sum_{n \geq m_j} a_n| \]
\[= |R(u_{m_j- 1}) - R(u_{m_j}) - a_{m_j- 1} (\lambda_{m_j- 1} - u_{m_j- 1})| \]
\[\leq e^{-\gamma_1 u_{m_j -1}} + e^{-\gamma_1 u_{m_j}} + e^{-\gamma_1 \lambda_{m_j}}\]
\[ \leq \Bigl(2 e^{\gamma_1(\lambda_{m_j} - u_{m_j - 1})} + 1\Bigr) e^{-\gamma_1 \lambda_{m_j}}.\]
Since 
\[\lambda_{m_j} - u_{m_j -1} = \lambda_{m_j} - \lambda_{m_j - 1} + \epsilon_j  < 3/2,\]
the above inequality yields
\[1 \leq \frac{1}{C}\Bigl(2 e^{\frac{3}{2} \gamma_1} +1\Bigr) e^{( - \gamma_1 + \delta + \gamma) \lambda_{m_j}}\]
and we obtain a contradiction for $\lambda_{m_j} \to \infty.$ This completes the proof. 

Since (\ref{eq:1.8}) is always satisfied for suitable frequencies  $\lambda_{m_j -1}, \lambda_{m_j}$ (see Section 4), exploiting (\ref{eq:1.11}), we may arrange the condition (\ref{eq:1.9}) for $\lambda_{m_j}$ large enough. An application of Theorem 1.1 yields Corollary 1.1.

\section{Intervals with clustering frequencies}

  We fix  $\delta > h + 2$ and $e^{-b} <\ep \ll 1/2$ and consider an interval $[b, b+1], \: b \geq  b_0(\ep).$ 
  Let $\lambda_k \in [ b + e^{-b}, b+ 1 - e^{-b}].$ 
  To examine the clustering around  $\lambda_k$, we construct some sets.  Introduce
 \[\J(\mu) = (\mu, \mu + \e^{- \delta b}).\]
 If $\lambda_{k+ 1}  \notin \J(\lambda_k),$ we stop the construction on the right. If $\lambda_{k+1} \in J_{\delta}(\lambda_k)$, one considers $J_{ \delta}(\lambda_{k+1}). $  In the case $\lambda_{k+2} \notin \J(\lambda_{k+1})$, we stop the construction. Otherwise, we continue with $\J(\lambda_{k+2})$ up to the situation when $\lambda_{k + q + 1} \notin \J(\lambda_{k+ q}).$ It is clear that such $q$ exists.We repeat the same construction moving on the left introducing
  \[\G(\mu) = (\mu- \e^{- \delta b}, \mu).\]
We stop when $\lambda_{k - p- 1} \notin \G(\lambda_{k-p}).$ Set
$I(\lambda_k, \delta) = [\lambda_{k - p}, \lambda_{k +q}].$
 The integers $p, q$ depend on $\lambda_k$, but we omit this in the notations below. Clearly, if we take another frequency $\lambda_{k'} \in I(\lambda_k, \delta),$ we obtain by the above construction the same interval. It is not excluded that $I(\lambda_k, \delta) = \{\lambda_k\}.$ In the particular case, one has $q = p = 0.$
 The number of the frequencies in $I(\lambda_k, \delta)$ is bounded by $e^{(h + \ep)(b + 1)}$
 and 
 \begin{equation} \label{eq:4.1}
  \lambda_{k+q} - \lambda_{k-p} \leq e^{(h - \delta+ \ep)b+ (h + \ep)} < e^{-b}, \: b \geq b_0(\ep).
 \end{equation}
 This estimate implies $\lambda_{k + q} < b+ 1,\: \lambda_{k - p} > b,$ so $I(\lambda_k, \delta) \subset (b, b+1).$
By Lemma 3.1, the intervals without frequencies have lengths less than $\ep$. Let $M(\ep, \delta, b)$ be the number of the sets
\[ I(\lambda_k, \delta) \cup (\lambda_{k + q}, \lambda_{k+ q + 1}), \: e^{-\delta b} \leq \lambda_{k + q + 1} -\lambda_{k+ q} < \ep, \: \lambda_k \in [b + e^{-b}, b+1 - e^{-b}].\]
Taking the union of such sets, we obtain
\[M(\ep, \delta, b) ( \ep + e^{- b}) \geq 1  - 2\ep- 2 e^{-b}.\]
 For large $b$ thus implies
\begin{equation} \label{eq:4.2}
 M(\ep,\delta, b) > \frac{1- 2\ep- 2 e^{-b}}{\ep + e^{-b}} = \frac{1}{\ep} - 2 + \mathcal O_{\ep}( e^{-b}).
\end{equation}
Hence we have at least $\Bigl[ \frac{1}{\ep}\Bigr] - 2$ frequencies $\lambda_{m_j}\in [b + e^{- b}, b+1 - e^{-b}]$ with
\begin{equation} \label{eq:4.3}
\lambda_{m_j - p_j} - \lambda_{m_j - p_j - 1} > e^{- \delta \lambda_{m_j - p_j - 1}},\: \lambda_{m_j + q_j + 1}- \lambda_{m_j + q_j} > e^{-\delta \lambda_{m_j + q_j}} ,
\end{equation}
where $[a]$ denotes the entire part of $a$.

Now let $\gamma \gg 1$ be fixed. Given an interval $I(\lambda_k, \delta) \subset (b, b+1),$ we have 2 possibilities:\\
          \[(i) \: \:\:|\sum_{n \geq k - p} a_n| \geq  e^{-\gamma \lambda_{k - p}},\:\:
         (ii) \: \:\:|\sum_{n \geq k - p} a_n| < e^{-\gamma \lambda_{k - p}}.\]  
         In the case (i) the conditions (\ref{eq:1.8}), (\ref{eq:1.9}) are satisfied for $\lambda_{k-p -1}$ and $\lambda_{k-p}.$ If one has (ii), and 
         $|\sum_{n \geq k+ q + 1} a_n| < e^{-\gamma \lambda_{k + q + 1}},$
          by triangle inequality one deduces
         \[|\sum_{n = k- p}^{k+q} a_n| \leq e^{- \gamma \lambda_{k-p}} + e^{- \gamma \lambda_{k+ q +1}} < 2 e^{-\gamma\lambda_{k-p}}.\]
         Thus if (ii) holds, and 
           $|\sum_{n = k- p}^{k+q} a_n|  \geq 2e^{-\gamma \lambda_{k- p}}$
           the conditions (\ref{eq:1.8}), (\ref{eq:1.9}) are satisfied for $\lambda_{k +q}$ and $\lambda_{k+ q + 1}.$           
                  Taking into account (\ref{eq:4.3}) and applying Theorem 1.1, we obtain the following
         \begin{corr} Suppose $\sigma_c < 0$. Suppose that there exist constants $\delta > h + 2, \gamma \gg 1$ and a sequence of intervals 
         \[I(\lambda_{m_j}, \delta) = [\lambda_{m_j - p_j}, \lambda_{m_j + q_j}], \: \lambda_{m_j} \nearrow \infty\]
         satisfying $(\ref{eq:4.3})$ 
          such that
          \[|\sum_{n = m_j- p_j}^{m_j+q_j } a_n|  \geq 2 e^{- \gamma \lambda_{m_j- p_j}}.\] 
          Then $\eta_D$ cannot be prolonged as entire function.       
         \end{corr}
It is important to increase the number of intervals included in $[b, b+ 1]$  satisfying (\ref{eq:4.3}). By using Lemma 3.1, we see that for $\ep\searrow 0$ we have $b_0(\ep) \nearrow \infty$ so a more precise asymptotics for the counting functions of the number of frequencies with remainder is necessary. Under some geometrical assumptions, it was proved (see \cite[Theorem 4]{petkov2012}) that we may replace $\ep$ by $ e^{- \mu b}$ with small $0 < \mu < h$  and  obtain a lower bound of
\[   \sharp\{ \gamma \in \Pi: \alpha \leq \tau^{\sharp}(\gamma) \leq \alpha + e^{-\mu b}\} .\] 
These assumptions are satisfied for $d = 2$, while for $d \geq 3$ one make some restrictions. We refer to \cite[Section 8]{petkov2012} for precise results and more details. Under these assumptions we  obtain $M(e^{-\mu b}, \delta, b) \sim e^{\mu b}$ as $b \to \infty $ so the number of intervals satisfying (\ref{eq:4.3}) increase exponentially as $b \to + \infty.$ The issue is that  the possibilities to satisfy the conditions of Theorem 1.1 increase exponentially, too.

To obtain a lower bound for $|a_n|, \: \forall n \geq n_0$, introduce the condition\\
    (L)  There exist constants $c_1 > 0, c_2 > 0,$ independent of $n$ such that 
   \begin{equation} \label{eq:4.4}
 |a_n| \geq c_1 e^{-c_2 \lambda_n},\:\forall n \geq  n_0.
\end{equation}   
    The  condition (\ref{eq:4.4}) holds in the case when the lengths of primitive periodic rays $\gamma \in \Pi$ are rationally independent, because (\ref{eq:1.7}) will contain only one term and from (\ref{eq:1.4}) one deduces (\ref{eq:4.4})  with $c_1 = \min_{i \neq j} {\rm dist}\: (D_i, D_j)$
  and $c_2 = d_2/2.$ This rational independence has been proved for generic domains  (see \cite[Theorem 6.2.3]{petkov2017geometry}). Then if (L) holds and 
  \begin{equation} \label{eq:4.5}
  |\sum_{k \geq m} a_k| < e^{-\gamma \lambda_m},\: |\sum_{k \geq m+1} a_k| < e^{- \gamma \lambda_{m+1}}
  \end{equation}
with $ \gamma > c_2 + 1$, one has $ c_1 e^{- c_2 \lambda_m} \leq  |a_m| <2 e^{-\gamma \lambda_m}$ which is impossible for large $\lambda_m.$ Hence at least one of the estimates (\ref{eq:4.5}) does not hold. We may study also the existence of 3 consecutive frequencies which are exponentially separated from each other.
\begin{corr} Assume $(L)$ satisfied. Suppose that there exist  constants $\delta > 0, \: C > 0$ and a sequence $\lambda_{m_j} \nearrow + \infty$ such that
\begin{equation} \label{eq:4.6}
\lambda_{m_j} - \lambda_{m_j - 1}  > Ce^{-\delta \lambda_{m_j}},\:  \lambda_{m_{j+1}} - \lambda_{m_j} > Ce^{-\delta \lambda_{m_{j+1}}}.
\end{equation} 
Then $\eta_D$ cannot be prolonged as entire function.    
\end{corr}
For the proof we exploit (\ref{eq:4.6}) and  we arrange easily (\ref{eq:1.9}). Then we apply Theorem 1. We conjecture that for generic domains there exists a sequence $\{\lambda_{m_j}\}$  satisfying (\ref{eq:4.6}).

Going back to intervals $I(\lambda_k, \delta)$, notice that for $\lambda_{k + q}$ one has also 2 possibilities:    
      \[(iii) \: \:\:|\sum_{n \geq k + q} a_n| \geq e^{-\gamma \lambda_{k + q}},\:\:
         (iv) \: \:\:|\sum_{n \geq k + q} a_n| <  e^{-\gamma \lambda_{k + q}}.\]           
 Assuming (L) and $\gamma > c_2 + 1$,  in the case (iv)  the conditions (\ref{eq:1.8}), (\ref{eq:1.9}) are satisfied for $\lambda_{k + q}$ and $\lambda_{k + q+ 1}.$   Consequently,  we obtain the following
 \begin{prop}   Assume $(L)$ satisfied and $\gamma > c_2 + 1$. Then for every interval $I(\lambda_k, \delta)$ we have $4$ possibilities:  $(i)-(iii), (i)-(iv), (ii)-(iii), (ii)-(iv).$ If $(i)$ holds, or if we have $(ii)-(vi)$, we may find an interval $[\lambda_{k-p- 1}, \lambda_{k- p}]$ or $[\lambda_{k + q}, \lambda_{k + q + 1}]$ satisfying $(\ref{eq:1.8})$ and $(\ref{eq:1.9}).$ 
\end{prop}
A more fine analysis of the estimates of the sums  $|\sum_{n = k_j -p_j}^{ k_j + q_j} a_n|$ should imply more precise results.

    \bibliographystyle{alpha}
\bibliography{bib.bib}

\end{document}